\documentclass[11pt]{amsart}

\usepackage{amsmath}
\usepackage{amsfonts}
\usepackage{amssymb}
\usepackage{amsthm}
\usepackage{mathtools}
\usepackage{url}
\usepackage{comment}
\usepackage{bbold}
\usepackage{appendix}
\usepackage{upgreek}
\usepackage{mathrsfs}

\usepackage{extsizes}
\usepackage[margin=3cm]{geometry}

\usepackage{enumitem}
\setenumerate[0]{label=(\roman*)}

\usepackage{hyperref}
 \hypersetup{
     colorlinks=true,
     linkcolor=red,
     filecolor=blue,
     citecolor = blue,      
     urlcolor=cyan,
     }

\DeclareMathOperator*{\EE}{\mathbb{E}}

\newcommand{\CC}{\mathbb{C}}
\newcommand{\NN}{\mathbb{N}}
\newcommand{\ZZ}{\mathbb{Z}}

\newcommand{\FF}{\mathbb{F}}

\newcommand{\1}{1}

\newcommand{\C}{\mathcal{C}}

\newcommand{\U}{\mathcal{U}}

\renewcommand{\P}{\mathcal{P}}

\newcommand{\h}{\mathfrak{h}}

\newcommand{\x}{{\rm{\textbf{x}}}}
\renewcommand{\v}{{\rm{\textbf{v}}}}

\renewcommand{\k}{\underline{k}}
\newcommand{\kk}{\underline{k}}
\renewcommand{\h}{{\underline{h}}}
\newcommand{\hh}{{\underline{h}}}
\newcommand{\yy}{{\underline{y}}}
\newcommand{\y}{{\underline{y}}}
\newcommand{\ww}{{\underline{w}}}
\newcommand{\w}{{\underline{w}}}

\newcommand{\Span}{{\rm{Span}}}
\renewcommand{\leq}{\leqslant}
\renewcommand{\geq}{\geqslant}

\renewcommand{\epsilon}{\varepsilon}

\theoremstyle{theorem}
\newtheorem{theorem}{Theorem}[section]

\newtheorem{proposition}[theorem]{Proposition}

\newtheorem{corollary}[theorem]{Corollary}

\newtheorem{lemma}[theorem]{Lemma}

\theoremstyle{definition}

\title[Multidimensional polynomial Szemer\'{e}di theorem in finite fields]{Multidimensional polynomial Szemer\'{e}di theorem in finite fields for polynomials of distinct degrees}
\author{Borys Kuca}
\date{}

\begin{document}

\maketitle

\begin{abstract}
    We obtain a polynomial upper bound in the finite-field version of the multidimensional polynomial Szemer\'{e}di theorem for distinct-degree polynomials. That is, if $P_1, ..., P_t$ are nonconstant integer polynomials of distinct degrees and $\v_1, ..., \v_t$ are nonzero vectors in $\FF_p^D$, we show that each subset of $\FF_p^D$ lacking a nontrivial configuration of the form
    $$\x, \x + \v_1 P_1(y), ..., \x + \v_t P_t(y)$$
    has at most $O(p^{D-c})$ elements. In doing so, we apply the notion of Gowers norms along a vector adapted from ergodic theory, which extends the classical concept of Gowers norms on finite abelian groups.
\end{abstract}

\section{Introduction}
We prove the following bound in the finite field version of the multidimensional polynomial Szemer\'{e}di theorem of Bergelson and Leibman \cite{bergelson_leibman_1996}.
\begin{theorem}\label{bounds}
Let $D, t\in\NN_+$, $\v_1, ..., \v_t\in\ZZ^D$ be nonzero vectors and $P_1, ..., P_t\in\ZZ[y]$ be polynomials satisfying $0 < \deg P_1 < ... < \deg P_t$. There exist constants $c,C>0$ and a threshold $p_0\in\NN$ such that for all primes $p>p_0$, each subset $A\subseteq \FF_p^D$ of size at least $C p^{D-c}$ contains
\begin{align}\label{configuration}
    \x,\; \x + \v_1 P_1(y), \; ...,\; \x + \v_t P_t(y)
\end{align}
for some $\x\in\FF_p^D$ and nonzero $y\in\FF_p$.
\end{theorem}
A special case of this statement is that each subset of $\FF_p^2$ of size $\Omega(p^{2-c})$ contains a nontrivial configuration of the form 
\begin{align}\label{square corners}
    (x_1, x_2), (x_1 + y, x_2), (x_1, x_2 + y^2),
\end{align}
previously proved in \cite{han_lacey_yang_2021},
or a novel result that each subset of $\FF_p^3$ of size $\Omega(p^{3-c})$ contains a nontrivial configuration of the form 
\begin{align}\label{cubic corners}
    (x_1, x_2, x_3), (x_1 + y, x_2, x_3), (x_1, x_2 + y^2, x_3), (x_1, x_2, x_3 + y^3). 
\end{align}
A configuration (\ref{configuration}), (\ref{square corners}) or (\ref{cubic corners}) is nontrivial if $y\neq 0$.

Theorem \ref{bounds} follows from the following result and its corollary. 
\begin{theorem}\label{count}
Let $D, t\in\NN_+$, $\v_1, ..., \v_t\in\ZZ^D$ be nonzero vectors and $P_1, ..., P_t\in\ZZ[y]$ be polynomials satisfying $0<\deg P_1 < ... < \deg P_t$. There exists $c>0$ and a threshold $p_0\in\NN$ such that for all primes $p>p_0$ and all 1-bounded functions $f_0, ..., f_t:\FF_p^D\to\CC$, we have
\begin{align*}
    \EE_{\substack{\x\in\FF_p^D,\\ y\in\FF_p}}f_0(\x)\prod_{i=1}^t f_i(\x+ \v_i P_i(y)) = \EE_{\x\in\FF_p^D}f_0(\x)\prod_{i=1}^t \EE_{n_i\in\FF_p}f_i(\x + \v_i n_i) + O(p^{-c}).
\end{align*}
\end{theorem}

\begin{corollary}\label{count for nonnegative functions}
Let $D, t\in\NN_+$, $\v_1, ..., \v_t\in\ZZ^D$ be nonzero vectors and $P_1, ..., P_t\in\ZZ[y]$ be polynomials satisfying $0<\deg P_1 < ... < \deg P_t$. There exists $c>0$ and a threshold $p_0\in\NN$ such that for all primes $p>p_0$ and all nonnegative 1-bounded functions $f:\FF_p^D\to\CC$, we have
\begin{align*}
    \EE_{\substack{\x\in\FF_p^D,\\ y\in\FF_p}}f(\x)\prod_{i=1}^t f(\x+ \v_i P_i(y)) \geq \left(\EE_{\x\in\FF_p^D} f(\x)\right)^{t+1}  + O(p^{-c}).
\end{align*}
\end{corollary}

A careful analysis of the proofs of Theorems \ref{bounds} and \ref{count} reveals that the bound in Theorem \ref{bounds} and the error terms in Theorem \ref{count} and Corollary \ref{count for nonnegative functions} can be chosen uniformly for all the nonzero vectors $\v_1, ..., \v_t$. Similarly, these quantities do not depend on the specific form of the polynomials $P_1, ..., P_t$, only on their degrees. The threshold $p_0$ in both theorems does however depend on the vectors and the polynomials. We also remark that both results hold for $\FF_q$ with $q$ being a prime power, provided that the characteristic of $\FF_q$ is sufficiently large in terms of the vectors $\v_1, ..., \v_t$ and the polynomials $P_1, ..., P_t$.

As an example, Theorem \ref{count} implies that
\begin{align*}
    &\EE_{x_1, x_2, y\in\FF_p}f_0(x_1, x_2) f_1(x_1 + y, x_2) f_2(x_1, x_2+y^2)\\
    &= \EE_{\substack{x_1, x_2,\\ x_1', x_2'\in\FF_p}} f_0(x_1, x_2) f_1(x_1', x_2) f_2(x_1, x_2') + O(p^{-c})
\end{align*}
for some $c>0$ uniformly in all 1-bounded functions $f_0, f_1, f_2: \FF_p^2\to\CC$. This particular statement has been proved in \cite{han_lacey_yang_2021} with an explicit constant $c = \frac{1}{8}$, but its natural analogue for (\ref{cubic corners}) is novel:
\begin{align}\label{count for 4-corners}
    \nonumber &\EE_{\substack{x_1, x_2, x_3,\\ y\in\FF_p}}f_0(x_1, x_2, x_3) f_1(x_1 + y, x_2, x_3) f_2(x_1, x_2+y^2, x_3)f_3(x_1, x_2, x_3 + y^3)\\ 
    &= \EE_{\substack{x_1, x_2, x_3,\\ x_1', x_2', x_3'\in\FF_p}} f_0(x_1, x_2, x_3) f_1(x_1', x_2, x_3) f_2(x_1, x_2', x_3) f_3(x_1, x_2, x_3') + O(p^{-c}).
\end{align}

It then follows from Corollary \ref{count for nonnegative functions} that if $A\subseteq\FF_p^3$ has size $|A|=\alpha p^3$ for $\alpha \gg p^{-c/4}$, then
\begin{align*}
    \left|\left\{\left((x_1, x_2, x_3), (x_1 + y, x_2, x_3), (x_1, x_2+y^2, x_3), (x_1, x_2, x_3 + y^3)\right)\in A^3\right\}\right|\leq \alpha^4 p^{12}.
\end{align*}

Theorems \ref{bounds} and \ref{count} extend results from \cite{han_lacey_yang_2021}, which proves the same statements in the special case $t=2$, $D=2$, $\v_1=(1,0)$ and $\v_2 = (0,1)$, i.e. for configurations of the form
\begin{align}
    (x_1, x_2), (x_1 + P_1(y), x_2), (x_1, x_2 + P_2(y)).
\end{align}
Theorems \ref{bounds} and \ref{count} also generalise results from the one-dimensional $(D=1)$ case \cite{bourgain_chang_2017, dong_li_sawin_2017, peluse_2018, peluse_2019, kuca_2020a}. Some of the results in the abovementioned papers also have integer analogues \cite{shkredov_2006a, shkredov_2006b, sarkozy_1978a, sarkozy_1978b, balog_pelikan_pintz_szemeredi_1994, slijepcevic_2003, lucier_2006, rice_2019, bloom_maynard_2020, prendiville_2017, peluse_prendiville_2019, peluse_prendiville_2020, peluse_2020}. In our paper, we develop multidimensional analogues of techniques pioneered in \cite{peluse_2019} and later used in \cite{peluse_2020, peluse_prendiville_2019, peluse_prendiville_2020, kuca_2020a}, and we use the version of the PET induction scheme from \cite{chu_frantzikinakis_host_2011}.

\subsection*{Notation}
Throughout the paper, we fix $D\in\NN_+$. We write elements of $\FF_p^D$ as $\x = (x_1, ..., x_D)$ and elements of $\FF_p$ as $x$. 

For a set $X$, we let $\EE_{x\in X} = \frac{1}{|X|}\sum_{x\in X}$ denote the average over $X$. If $X = \FF_p^D$ or $\FF_p$, then we suppress the mentioning of the set and let $\EE_\x = \EE_{\x\in\FF_p^D}$ and $\EE_x = \EE_{x\in\FF_p}$. Given a vector $\v_i\in\FF_p^D$, we denote $V_i = \Span_{\FF_p}\{\v_i\}$, and we define $\EE(f|V_i)(\x) = \EE_{\x + V_i}f = \EE_{y}f(\x + \v_i y)$ to be the average of $f$ along the coset $\x+V_i$. We also set $\v_0 = \textbf{0}$ and $P_0 = 0$, and we call a function $f:\FF_p^D\to\CC$ \emph{1-bounded} if $\|f\|_{\infty}=\max_\x|f(\x)|\leq 1$. Finally, we set $\|f\|_s = \left(\EE_\x |f(\x)|^s\right)^\frac{1}{s}$ for $1\leq s < \infty$.

We begin with specifying further pieces of notation used in this paper. For $k\in\ZZ$ and $P\in\ZZ[y]$, we set $\partial_k P(y) = P(y)-P(y+k)$, and for $f:\FF_p^D\to\CC$, we let $\Delta_\v f(\x) = f(\x)\overline{f(\x+\v)}$. We also let $\C z = \overline{z}$ be the conjugation operator. For $\underline{w}\in\{0,1\}^s$, we set $|w|=w_1 + ... + w_s$. Finally, the letter $p$ denotes a (sufficiently large) prime, and we let $e_p(x) = e^{2\pi i x/p}$.

We use the asymptotic notation in the standard way. If $I\subseteq\NN$ and $f,g:\NN\to\CC$, with $g$ taking positive real values, we denote $f=O(g)$, $f\ll g$, $g\gg f$ or $g = \Omega(f)$ if there exists $C>0$ such that $|f(n)|\leq C g(n)$ for all $n\in I$. If the constant $C$ depends on a paramter, we record this dependence with a subscript. All constants are allowed to depend on $D, t$, the polynomials $P_1, ..., P_t$ or the vectors $\v_1, ..., \v_t$, and this dependence is not recorded. An exception to this rule are results in Section \ref{section on PET}, where we specify all the parameters that constants depend on. 

\subsection*{Acknowledgments}
We would like to thank Sean Prendiville for useful conversations and comments on an earlier version of this paper, Donald Robertson for consultations on this project, and an anonymous referee for their detailed suggestions. 

\section{Gowers norms along a vector}
To prove Theorem \ref{count}, we need a notion of Gowers norms along a vector. Let $f:\FF_p^D\to\CC$, $s\in\NN_+$ and $\v\in\ZZ^D$. We define the Gowers norm of $f$ of degree $s$ along $\v$ to be
\begin{align*}
    \|f\|_{U^s(\v)}=\left(\EE_{\x, h_1, ..., h_s}\prod_{w\in\{0,1\}^s}\C^{|w|}f(\x + \v(w_1 h_1 + ... + w_s h_s))\right)^\frac{1}{2^s}.
\end{align*}
These norms are finitary analogues of Host-Kra seminorms from ergodic theory, corresponding to the transformation $T\x = \x + \v$ on $\FF_p^D$. If $D=1$, then the norm $U^s(\v)$ equals the 1-dimensional Gowers norm $U^s$ for any nonzero vector $\v$. If $D=2$ and $\v = (1,0)$, we have
\begin{align*}
    \|f\|_{U^2(\v)}=\left(\EE_{\substack{x_1, x_2,\\ h_1, h_2}} f(x_1, x_2)\overline{f(x_1+h_1, x_2) f(x_1 + h_2, x_2)} f(x_1+h_1 + h_2, x_2) \right)^{\frac{1}{4}}.
\end{align*}

Gowers norms along a vector satisfy a lot of the usual properties of Gowers norms. Letting $f_\x(n) = (\x + \v n)$, we can relate Gowers norms along a vector to 1-dimensional Gowers norms via the formula
\begin{align}\label{relation between Gowers norms}
    \|f\|_{U^s(\v)} = \left(\EE_\x \|f_\x\|_{U^s}^{2^s}\right)^\frac{1}{2^s}.
\end{align}
The $U^1(\v)$ norm is in fact a seminorm, and it is given by
\begin{align*}
    \|f\|_{U^1(\v)}^2 &= \EE_{\x, h}f(\x)\overline{f(\x + \v h)} = \EE_{\x, h, k}f(\x + \v k)\overline{f(\x + \v h)}\\ 
    &= \EE_{\x}\left|\EE_{h}f(\x + \v h)\right|^2 = \EE_{\x}\left|\EE_{\x + V} f\right|^2 = \|\EE(f|V)\|_{2}^2.
\end{align*}
Having large $U^1(\v)$ norm thus tells us that $f$ has large average on many cosets of $V = \Span_{\FF_p}\{\v\}$. In particular, for $D = 2$ and $\v = (1,0)$, we have
\begin{align*}
    \|f\|_{U^1(\v)}^2 = \EE_{x_2}\left|\EE_{x_1} f(x_1, x_2)\right|^2.
\end{align*}

The identity $\|f\|_{U^1(\v)} = \|\EE(f|V)\|_{L^{2}}$ can be extended to higher values of $s$ as follows: if $f$ is $V$-measurable in the sense of being constant on cosets of $V$, then $\|f\|_{U^s(\v)} = \|\EE(f|V)\|_{L^{2^s}}$. 

For $s\geq 2$, the seminorm $U^s(\v)$ is a norm and satisfies the usual monotonicity property
\begin{align*}
    \|f\|_{U^1(\v)}\leq \|f\|_{U^2(\v)}\leq \|f\|_{U^3(\v)}\leq ...
\end{align*}
Both of these properties can be derived from the formula (\ref{relation between Gowers norms}) and the corresponding properties for 1-dimensional Gowers norms. It is also straightforward to deduce from the definition that $U^s(\v)$ satisfies the induction property
\begin{align*}
    \|f\|_{U^s(\v)} = \left(\EE_{h_{k+1}, ..., h_s}\|\Delta_{\v h_{k+1}, ..., \v h_s} f\|_{U^k(\v)}^{2^k}\right)^\frac{1}{2^s}. 
\end{align*}

We need a better understanding of the $U^2(\v)$ norm. This norm can be related to Fourier analysis as follows. For $\x, \v\in\FF_p^D$ and $k\in\FF_p$, we define the Fourier transform of $f$ along $\v$ as
\begin{align*}
    \widehat{f}(\x; \v; k) = \EE_{n}f(\x + \v n) e_p(-k n),
\end{align*}
so that 
\begin{align}\label{inverse formula for Fourier transform}
    f(\x + \v n) = \sum_{k\in\FF_p}\widehat{f}(\x; \v; k) e_p(k n)
\end{align}
for any $n\in\FF_p$. There is an ambiguity involved in the formula (\ref{inverse formula for Fourier transform}), coming from expressing the value of $f$ at $\x+\v n$ in terms of the value of $\hat{f}$ at $\x$. One can however check that (\ref{inverse formula for Fourier transform}) works if we replace $\x$ and $\x + \v n$ by any $\x'$ and $\x'+\v n'$. 

In particular, letting $D = 2$ and $\v = (1,0)$, we get
\begin{align*}
    \widehat{f}(\x; \v; k) = \EE_{n}f(x_1 + n, x_2) e_p(-k n),
\end{align*}
We observe that $|\widehat{f}(\x; \v; k)|=|\widehat{f}(\x'; \v; k)|$ whenever $\x-\x'\in V$. 
With these definitions, we have
\begin{align}\label{U^2 inverse}
    \| f\|_{U^2(\v)}^4 = \EE_{\x, k}|\widehat{f}(\x;\v; k)|^4\leq \EE_{\x} |\widehat{f}(\x; \v; \phi(\x))|^2
\end{align}
whenever $f$ is 1-bounded, where $\phi(\x)$ is an element of $\FF_p$ for which $\max_{k} |\widehat{f}(\x; \v; k)| = |\widehat{f}(\x; \v; \phi(\x))|$. 

The Fourier transform can also be used to give an alternative description of the $U^1(\v)$ norm; specifically, $\|f\|^2_{U^1(\v)} = \EE_\x|\widehat{f}(\x; \v; 0)|^2$ since $\widehat{f}(\x; \v; 0) = \EE(f|V)(\x)$.

We also need the following variant of the classical exponential sums estimates, which can be found e.g. in \cite{kowalski_2018} as Theorem 3.2.
\begin{lemma}\label{Weyl sums}
Let $P\in\ZZ[y]$ be a polynomial with $\deg P = d$ satisfying $1\leq d < p$. Then
\begin{align*}
    \left|\EE_{y} e_p(P(y))\right| \leq (d-1) p^{-1/2}.
\end{align*}
\end{lemma}

\section{The outline of the argument}
We fix integers $0\leq m\leq t$, nonzero vectors $\v_1, ..., \v_t\in\ZZ^D$, and polynomials $P_1, ..., P_t\in\ZZ[y]$ satisfying $0< \deg P_1 < \deg P_2 < ... < \deg P_t$.

Roughly speaking, our proof of Theorem \ref{bounds} goes by induction on $t$, and it follows the three-step strategy of \cite{peluse_2019}. Like in \cite{peluse_2019}, we start by obtaining a global Gowers norm control on the operator
\begin{align}\label{operator}
    \EE_{\x, y}\prod_{i=0}^t f_i(\x+ \v_i P_i(y)). 
\end{align}
We then perform a degree-lowering argument to show that we can in fact control this operator by a $U^1$-type norm. Finally we use the properties of this norm to show that 
\begin{align*}
    \left|\EE_{\x, y}\prod_{i=0}^t f_i(\x+ \v_i P_i(y)) - \EE_{\x}\prod_{i=0}^t \EE(f_i|V_i)(\x)\right| \ll p^{-c}
\end{align*}
for some $c>0$. 

One difference between our argument and that of \cite{peluse_2019} is that in contrast to the $D=1$ case, where we would control (\ref{operator}) by Gowers norms of each of the function $f_0, ..., f_t$, in the $D>1$ case we can only bound (\ref{operator}) in terms of some Gowers norm $U^s(\v_t)$ of the function $f_t$. Moreover, obtaining such a bound is only possible under the extra assumption that $\deg P_t > \max(\deg P_1, ..., \deg P_{t-1}, 0)$, whereas the $D=1$ case only requires linear independence of $P_1, ..., P_t$. The PET induction procedure that produces such a bound has been developed in \cite{chu_frantzikinakis_host_2011}, and we adapt the results of this paper in Section \ref{section on PET}.

In the $D=1$ case, the $U^1$ Gowers norm is of ``$L^1$ type", in the sense that $\|f\|_{U^1} = |\EE_{x} f(x)|\leq\|f\|_{1}$ for any $f:\FF_p\to\CC$. However, in the $D>1$ case the $U^1(\v)$ norm is of ``$L^2$ type" for any nonzero vector $\v\in\FF_p^D$, in the sense that $\|f\|_{U^1(\v)}=\|\EE(f|V)\|_{2}$. As a consequence, it turns out that we need to obtain a Gowers norm control of the $L^2$ norm of the function
\begin{align*}
    G_t(\x) = \EE_y\prod_{i=1}^t f_i(\x+ \v_i P_i(y)) 
\end{align*}
rather than the operator (\ref{operator}). Together with an application of the Cauchy-Schwarz inequality, this implies the Gowers norm control of (\ref{operator}). 

To be able to perform induction on $t$, we need to consider more general operators
\begin{align*}
    G_{m,t}(\x) = \EE_y\prod_{i=1}^m f_i(\x+ \v_i P_i(y)) \prod_{i=m+1}^t e_p(\phi_i(\x)P_i(y))1_\U(\x)
\end{align*}
for some phase functions $\phi_{m+1}, ..., \phi_t:\FF_p^D\to\FF_p$ and $\U\subseteq\FF_p^D$. By applying a trick from Lemma 5.12 of \cite{prendiville_2020} and a variant of the PET induction procedure outlined in Section \ref{section on PET}, we show that this operator is controlled by the $U^s(\v_m)$ of the dual function
\begin{align*}
    F_{m,t}(\x) = &\EE_{y,k} \left(\prod_{i=1}^{m-1}f_i(\x+\v_i P_i(y)-\v_m P_m(y+k))\overline{f_i(\x+\v_i P_i(y+k)-\v_m P_m(y+k))}\right)\\
    &f_m(\x+\v_m (P_m(y)- P_m(y+k)))\left(\prod_{i=m+1}^t e_p(\phi_i(\x -\v_m P_m(y+k)) \partial_k P_i(y))\right)1_\U(\x-\v_m P_m(y+k)).
\end{align*}
A degree-lowering argument then shows us that the $U^s(\v_m)$ norm of $F_{m,t}$ can be bounded from above by a small power of $\|F_{m,t}\|_{U^1(\v_m)}$ and an error term $O(p^{-c})$. This norm can in turn be bounded from above by the norms $\|f_i\|_{U^1(\v_i)}$ for $1\leq i\leq m$, from which we deduce that 
\begin{align*}
    \|G_{m,t}\|_{2}^2 = \EE_{\x}\left|\prod_{i=1}^m \EE(f_i|V_i)(\x)\right|^2 \prod_{i=m+1}^t 1_{\phi_i(\x) = 0} 1_\U(\x) + O(p^{-c})
\end{align*}
Theorem \ref{count} follows by taking $m=t$ and $\U=\FF_p^D$.

The proof that the $L^2$ norm of $G_{m,t}$ is controlled by  the norms $\|f_i\|_{U^1(\v_i)}$ for $1\leq i\leq m$, with all the degree-lowering arguments, occupies the entirety of Section \ref{section on computations}. In Section \ref{section on bounds}, we conclude the proof of Theorem \ref{count} and use it to derive Theorem \ref{bounds}.

\section{Controlling counting operators by Gowers norms}\label{section on PET}
The material in this section follows closely Sections 4 and 5 of \cite{chu_frantzikinakis_host_2011}. We say that two nonconstant polynomials $P, Q\in\ZZ[y]$ are equivalent, denoted $P\sim Q$, if they have the same degree and the same highest-degree coefficient; equivalently, $P\sim Q$ iff $\deg(P-Q)<\min\{\deg P,\; \deg Q\}$. 

Let $t, m\in\ZZ$ and $\P_j = (P_{j1}, ..., P_{jm})\in\ZZ[y]^m$ for $1\leq j\leq t$. 
We want to determine when the operator
\begin{align}\label{general counting operator}
    \EE_{\x, y} f_0(\x) \prod_{i=1}^m f_i(\x + \v_1 P_{1i}(y) + ... + \v_t P_{ti}(y))
\end{align}
is controlled by a Gowers norm for some nonzero vectors $\v_1, ..., \v_t\in\ZZ^D$, and the tuple $\P=(\P_1, ..., \P_t)$ is a compact way of encoding information about the polynomials appearing in (\ref{general counting operator}). 

Let $d=\max\limits_{j, i} \deg P_{ji}$. We define $$\P_j' = \{P_{ji}: \deg P_{j'i} = 0\; {\rm{for}}\; j<j'\leq t\},$$
and we let $w_{jk}$ be the number of distinct equivalence classes of polynomials of degree $k$ in $\P_j'$. The \emph{type} of the family $(\P_1, ..., \P_t)$ is the matrix
\begin{align*}
    \begin{pmatrix}
    w_{11} & \hdots & w_{1d}\\
    w_{21} & \hdots & w_{2d}\\
    \vdots & \hdots & \vdots\\
    w_{t1} & \hdots & w_{td}.
    \end{pmatrix}
\end{align*}

Given two $t\times d$ matrices $W=(w_{jk})$ and $W'=(w'_{jk})$, we order them in the reversed lexicographic way; that is, $W<W'$ if $w_{td}<w'_{td}$, or $w_{td}=w'_{td}$ and $w_{t(d-1)}<w'_{t(d-1)}$, ..., or $w_{tk} = w_{tk}$ for all $1\leq k\leq d$ and $w_{(t-1)d}<w'_{(t-1)d}$, and so on. 

The family $(\P_1, ..., \P_t)$ is \emph{nice} if 
\begin{enumerate}
    \item $\deg P_{tm} \geq \deg P_{ti}$ for $1\leq i\leq m$;
    \item $\deg P_{tm} > \deg P_{ji}$ for $1\leq j < t$ and $1\leq i\leq m$ if $t>1$, and $\deg P_{tm}\geq 1$ if $t = 1$. 
    \item $\deg(P_{tm} - P_{ti}) > \deg(P_{jm}-P_{ji})$ for  $1\leq j < t$ and $1\leq i< m$.
\end{enumerate}

The arguments from Sections 4 and 5 of \cite{chu_frantzikinakis_host_2011}, after appropriate adaptations to the finite field setting, can be used to show the following.
\begin{proposition}\label{general Gowers norm control}
Let $m,t, d\in\NN_+$ and $\P = (\P_1, ..., \P_t)$ be a nice family of polynomials of degree $d$. There exist $s\in\NN_+$ and $C, c>0$ depending only on $m, t, d$ such that for any nonzero vectors $\v_1, ..., \v_t\in\ZZ^D$ and any 1-bounded functions $f_0, ..., f_m:\FF_p^D\to\CC$, we have the bound
\begin{align*}
    \left|\EE_{\x, y} f_0(\x) \prod_{i=1}^m f_i(\x + \v_1 P_{1i}(y) + ... +\v_t P_{ti}(y))\right|\leq \|f_m\|_{U^s(\v_t)}^c + C p^{-c}.
\end{align*}
\end{proposition}
To prove Proposition \ref{general Gowers norm control}, we need a 1-dimensional estimate, which is a version of the classical \emph{generalised von Neumann theorem}. We advise the reader to consult \cite{green_2007} so as to see how statements like this can be proved.
\begin{lemma}\label{generalised von Neumann}
Let $m\in\NN_+$ $f_0, ..., f_m:\FF_p\to\CC$ be 1-bounded and $0< a_m < p$ be distinct from $a_1, ..., a_{m-1}$. Then 
\begin{align*}
    \left|\EE_{x,y}f_0(x)f_1(x+a_1 y)\cdots f_m(x + a_m y)\right|\leq \|f_m\|_{U^m}.
\end{align*}
\end{lemma}
\begin{proof}[Proof of Proposition \ref{general Gowers norm control}]
Suppose first that $\deg P_{tm} = 1$. By the definition of nice families of polynomials, this means that $P_{ti}(y) = a_i y + b_i$ for integers $a_{1}, ..., a_{m}, b_1, ..., b_m$, and $P_{ji}$ is constant for all $1\leq j < t$. Translating each $f_i$ by the constant expression $\v_1 P_{1i}(y) + ... + \v_{t-1} P_{(t-1)i}(y)$ if necessary, we end up with the expression
\begin{align*}
    \EE_{\x, y} f_0(\x) f_1(\x + \v_t (a_{1}y + b_1)) \cdots f_m(\x + \v_t (a_{m}y+b_m)).
\end{align*}
The assumption $\deg P_{tm} = 1$ and the condition (iii) in the definition of nice families imply that
\begin{align}\label{distinctness of a_m}
    a_m \notin \{0, a_1, ..., a_{m-1}\};   
\end{align}
otherwise $\deg (P_{tm} - P_{ti}) = 0 = \deg (P_{jm} - P_{ji})$ for some $0\leq i < m$ and all $1 \leq j < t$.

We let $f_{i,\x}(y) = f_i(\x + \v_t (y + b_i))$, so that
\begin{align*}
    &\left|\EE_{\x, y} f_0(\x) f_1(\x + \v_t (a_{1}y + b_1)) \cdots f_m(\x + \v_t (a_{m}y + b_m))\right|\\
    &= \left|\EE_{\x, n, y} f_0(\x) f_1(\x + \v_t(n+a_1 y + b_1)) \cdots f_m(\x + \v_t (n + a_m y + b_m))\right|\\
    &\leq \EE_{\x}\left|\EE_{n,y}  f_{1, \x}(n+ a_1 y)\cdots f_{m, \x}(n+ a_m y)  \right|.
\end{align*}

Using Lemma \ref{generalised von Neumann} together with (\ref{distinctness of a_m}), and assuming by a compactness argument that $p$ is large enough with respect to $a_m$, we deduce that
\begin{align*}
    \left|\EE_{n,y}  f_{1, \x}(n+ a_1 y)\cdots f_{m, \x}(n+ a_m y)  \right|\leq \|f_{m,\x}\|_{U^m}.
\end{align*}
We then apply the relation (\ref{relation between Gowers norms}) and the H\"{o}lder inequality to conclude that
\begin{align*}
    \left|\EE_{\x, y} f_0(\x) f_1(\x + \v_t a_{1}y) \cdots f_m(\x + \v_t a_{m}y)\right|\leq \EE_\x \|f_{m,\x}\|_{U^{m}}\leq \|f_m\|_{U^{m}(\v_t)}, 
\end{align*}
which finishes the proof in the $d=1$ case. 

Suppose now that $\deg P_{tm} = d > 1$. We can assume that for each $1\leq i\leq m$, the polynomial map $y\mapsto \v_1 P_{1i}(y) + ... + \v_t P_{ti}(y)$ is nonconstant, otherwise we incorporate $f_i$ into $f_0$. We
proceed in three steps. First, we apply the Cauchy-Schwarz inequality to bound
\begin{align*}
    &\left|\EE_{\x, y} f_0(\x) \prod_{i=1}^m f_i(\x + \v_1 P_{1i}(y) + ... +\v_t P_{ti}(y))\right|^2\\
    &\leq \left|\EE_{\x, y, h} \prod_{i=1}^m f_i(\x + \v_1 P_{1i}(y) + ... +\v_t P_{ti}(y))\overline{f_i(\x + \v_1 P_{1i}(y+h) + ... +\v_t P_{ti}(y+h))}\right|.
\end{align*}
Second, we translate $\x\mapsto \x - \v_1 Q_1(y) - ... - \v_t Q_t(y)$ for polynomials $Q_1, ..., Q_t\in\ZZ[y]$ to be chosen later, set $\tilde{P}_{ji; h}(y) = P_{ji}(y+h) - Q_i(y)$ and use the triangle inequality to bound the expression above by
\begin{align}\label{inequality in PET}
    \EE_h \left|\EE_{\x, y} \prod_{i=1}^m f_i(\x + \v_1 \tilde{P}_{1i;0}(y) + ... +\v_t \tilde{P}_{ti;0}(y))\overline{f_i(\x + \v_1 \tilde{P}_{1i;h}(y) + ... +\v_t \tilde{P}_{ti;h}(y))}\right|.
\end{align}
We choose the polynomials $Q_1, ..., Q_t$ in such a way that for all except at most $m-1$ differences $h\in\ZZ$, the family $S_h\P = (S_h\P_1, ..., S_h\P_t)$, where $S_h\P_j = (\tilde{P}_{j1;0}, \tilde{P}_{j1; h}, ..., \tilde{P}_{jm;0}, \tilde{P}_{jm;h})$, is nice and has a type strictly smaller than $\P$. The procedure of picking appropriate $Q_1, ..., Q_t$ goes the same way as in Lemma 5.4 of \cite{chu_frantzikinakis_host_2011}; we restate here the algorithm from that paper for completeness. Let $l = \min\limits_j\{\P_j' {\rm{\;is\; nonempty}}\}$. If $l<t$, then we take $Q_l$ to be a polynomial of the smallest degree in $\P_l'$ and set $Q_j = 0$ for $j\neq l$. Then the $l$-th row of the type matrix gets reduced while the rows indexed $l+1, ..., t$ remain unchanged. 

If $l=t$, i.e. $\P_1', ..., \P_{t-1}'$ are all empty, then we split into two cases. If 
$P_{ti}\sim P_{tm}$ for all $1\leq i \leq m$, then we set $Q_j = P_{jm}$ for all $1\leq j\leq t$. In this case, $w_{td}$ decreases from 1 to 0, and so we obtain a strictly smaller type matrix. Otherwise we choose $1\leq i\leq m$ such that $P_{ti}$ has the smallest degree of all $P_{t1}, ..., P_{tm}$ and let $Q_j = P_{ji}$ for all $1\leq j\leq t$. This reduces the $t$-th row of the type matrix by one. 



Thus, the family $S_h \P$ has a strictly smaller type than $\P$ for all $h\in\ZZ$. Lemmas 4.4 and 5.4 of \cite{chu_frantzikinakis_host_2011} further show that for all except at most $m-1$ values of $h\in\ZZ$, the family $S_h\P$ is nice and has a strictly smaller type than $\P$. The proof of the niceness of $S_h\P$ in \cite{chu_frantzikinakis_host_2011} uses the assumption $d\geq 2$, hence the necessity to distinguish the cases $d = 1$ and $d \geq 2$.

We complete the proof of Proposition \ref{general Gowers norm control} in the case $d\geq 2$ by induction on the type of the progression using the $d=1$ case as the base case. By Lemma 4.5 and 5.5 of \cite{chu_frantzikinakis_host_2011}, there exists a constant $M=M(d, m, t)$ independent of the choice of the polynomial family $\P$ or vectors $\v_1, ..., \v_t$, such that after repeating the abovementioned procedure of applying the Cauchy-Schwarz inequality and performing the change of variables at most $M$ times to a polynomial family $\P$, we end up with a nice polynomial family of degree 1. Letting $H$ be the set of $h$ such that $S_h \P$ is nice, and using the fact that the exceptional set $\FF_p \setminus{H}$ has at most $m-1$ elements, we obtain the bound
\begin{align*}
    &\left|\EE_{\x, y} f_0(\x) \prod_{i=1}^m f_i(\x + \v_1 P_{1i}(y) + ... +\v_t P_{ti}(y))\right|^2\\
    &\leq \EE_h\left|\EE_{\x, y} \prod_{i=1}^m f_i(\x + \v_1 P_{1i}(y) + ... +\v_t P_{ti}(y))\overline{f_i(\x + \v_1 P_{1i}(y+h) + ... +\v_t P_{ti}(y+h))}\right|\cdot 1_H(h) + (m-1) p^{-1}.
\end{align*}
We now apply the induction hypothesis to the families $(S_h \P)_{h\in H}$, which by assumption are nice and have type strictly less than $\P$, to conclude that there exist $c_0, C_0, s > 0$, independent of $h$, such that
\begin{align*}
    \left|\EE_{\x, y} \prod_{i=1}^m f_i(\x + \v_1 \tilde{P}_{1i;0}(y) + ... +\v_t \tilde{P}_{ti;0}(y))\overline{f_i(\x + \v_1 \tilde{P}_{1i;h}(y) + ... +\v_t \tilde{P}_{ti;h}(y))}\right| \leq \|f_m\|_{U^s}^{c_0} + C_0 p^{-c_0}
\end{align*}
for every $h\in H$. It follows that 
\begin{align*}
    \left|\EE_{\x, y} f_0(\x) \prod_{i=1}^m f_i(\x + \v_1 P_{1i}(y) + ... +\v_t P_{ti}(y))\right| &\leq \left(\|f_m\|_{U^s}^{c_0} + C_0 p^{-c_0} + (m-1)p^{-1}\right)^\frac{1}{2}\\
    &\leq \|f_m\|_{U^s}^{c} + C p^{-c}
\end{align*}
for some $c, C>0$ that only depend on $m, d, t$. 
\end{proof}

We will amply use the following corollary of Proposition \ref{general Gowers norm control}, which plays the same role as Proposition 2.2 of \cite{peluse_2019} in that paper. 
\begin{corollary}\label{Gowers norm control of twisted operators}
Let $1\leq m\leq t$ be integers and $P_1, ..., P_t\in\ZZ[y]$ be polynomials satisfying $0< \deg P_1 < \deg P_2 < ... < \deg P_m$.
There exist $s\in\NN_+$ and $c>0$ such that for any nonzero vectors $\v_1, ..., \v_m\in\ZZ^D$, 1-bounded functions ${f_0, ..., f_m, g_1, ..., g_m:\FF_p^D\to\CC}$, and phase functions $\phi_{m+1}, ..., \phi_t:\FF_p^D\to\FF_p$, we have the bound
\begin{align*}
    \left|\EE_{\x, y,k}f_0(\x)\prod_{i=1}^{m}f_i(\x+\v_i P_i(y))\overline{g_i(\x+\v_i P_i(y+k))}\prod_{i = m+1}^t e_p(\phi_i(\x)\partial_k P_i(y)) \right|\leq \|g_m\|_{U^s(\v_m)}^c + O(p^{-c}).
\end{align*}
\end{corollary}
\begin{proof}
By applying the Cauchy-Schwarz inequality in $\x$ and $k$ to
\begin{align}\label{counting operator with phases}
    \EE_{\x, y,k}f_0(\x)\prod_{i=1}^{m}f_i(\x+\v_i P_i(y))\overline{g_i(\x+\v_i P_i(y+k))}\prod_{i = m+1}^t e_p(\phi_i(\x)\partial_k P_i(y))
\end{align}
and setting $k_1 = k$, we observe that
\begin{align*}
    &\left|\EE_{\x, y,k}f_0(\x)\prod_{i=1}^{m}f_i(\x+\v_i P_i(y))\overline{g_i(\x+\v_i P_i(y+k))}\prod_{i = m+1}^t e_p(\phi_i(\x)\partial_k P_i(y)) \right|^2\\
    &\leq \EE_{\substack{\x, y,\\ k_1, k_2}}\prod_{i=1}^{m}\prod_{\w\in\{0,1\}^2} \C^{|w|} g_i^{(\w)}(\x+\v_i P_i(y+w_1 k_1 + w_2 k_2))\prod_{i = m+1}^t e_p(\phi_i(\x)\partial_{k_1, k_2} P_i(y)), 
\end{align*}
where $g_i^{(0,0)} = g_i^{(0,1)} = f_i$ and $g_i^{(1,0)} = g_i^{(1,1)} = g_i$.
Importantly, the degree of the polynomial $\partial_{k_1, k_2} P_i$ in $y$ is 1 less than the degree of $\partial_{k_1} P_i$. If $d = \max\{\deg P_{m+1}, ..., \deg P_t\}$, then we get rid of the phases $e_p(\phi_i(\x)\partial_k P_i(y))$ by applying the Cauchy Schwarz inequality $d+1$ times in all variables except $y$ to (\ref{counting operator with phases})  in a similar fashion. Thus,
\begin{align}\label{Cor 4.3, 1}
    \nonumber 
    &\left|\EE_{\x,y,k}f_0(\x)\prod_{i=1}^{m}f_i(\x+\v_i P_i(y))\overline{g_i(\x+\v_i P_i(y+k))}\prod_{i = m+1}^t e_p(\phi_i(\x)\partial_k P_i(y)) \right|^{2^{d+2}}\\
    &\leq \EE_{\substack{\x, y,\\ k_1, ..., k_{d+2}}}\prod_{i=1}^{m}\prod_{\w\in\{0,1\}^{d+2}} \C^{|w|} g^{(\w)}_i(\x+\v_i P_i(y+w_1 k_1 + ... + w_{d+2} k_{d+2})), 
\end{align}
where $g^{(\w)}_i = f_i$ if $w_1 = 0$ and $g^{(\w)}_i = g_i$ otherwise.

We now split into the cases $\deg P_m = 1$ and $\deg P_m >1$ and start with the former. In this case, the assumption $0 < \deg P_1 < ... < \deg P_m$ necessitates $m = 1$. Letting $P_m(y) = ay + b$, we rewrite (\ref{Cor 4.3, 1}) as
\begin{align*}
    \EE_{\substack{\x, y,\\ k_1, ..., k_{d+2}}}\prod_{\w\in\{0,1\}^{d+2}} \C^{|w|} g_{1}^{(\w)}(\x+ a \v_1(y+ w_1 k_1 + ... + w_{d+2} k_{d+2}))\\
    = \EE_{\substack{\x, y,\\ k_1, ..., k_{d+2}}}\prod_{w_2, ..., w_{d+2}\in\{0,1\}} \C^{|w|} (f_1(\x+ a \v_1(y + w_2 k_2 + ... + w_{d+2} k_{d+2}))\\
    \overline{g_1(\x+ a \v_1(y + k_1 + w_2 k_2 + ... + w_{d+2} k_{d+2}))}).
\end{align*}
Substituting $y \mapsto y/a, k_1 \mapsto k_1/a, ..., k_{d+2}\mapsto k_{d+2}/a$, applying the Cauchy-Schwarz inequality in $k_1$, and performing a change of variables, we bound the right-hand side of the expression above by
\begin{align*}
    \left(\EE_{\substack{\x, y,\\ k_1, ..., k_{d+2}}}\prod_{\w\in\{0,1\}^{d+2}} \C^{|w|} g_1(\x+ \v_1(y + k_1 + ... + k_{d+2}))\right)^\frac{1}{2} = \|g_1\|_{U^{d+2}(\v_1)}^{2^{d+1}}.
\end{align*}
Thus,
\begin{align*}
    \left|    \EE_{\x, y,k}f_0(\x)\prod_{i=1}^{m}f_i(\x+\v_i P_i(y))\overline{g_i(\x+\v_i P_i(y+k))}\prod_{i = m+1}^t e_p(\phi_i(\x)\partial_k P_i(y))\right|\leq \|g_m\|_{U^{d+2}(v_1)}^\frac{1}{2}
\end{align*}
whenever $\deg P_m = 1$. 

We now return to the case $\deg P_m >1$. For $(1-O(p^{-1}))p^{d+2}$ values $(k_1, ..., k_{d+2})\in\FF_p^{d+2}$, the expressions $(w_1 k_1 + ... + w_{d+2} k_{d+2})_{\w\in\{0,1\}^{d+2}}$ are all distinct. By (\ref{Cor 4.3, 1}) and the pigeonhole principle, there exists a tuple $(k_1, ..., k_{d+2})\in\FF_p^{d+2}$ satisfying this property, for which moreover
\begin{align*}
    &\left|\EE_{\x,y,k} f_0(\x) \prod_{i=1}^{m}f_i(\x+\v_i P_i(y))\overline{g_i(\x+\v_i P_i(y+k))}\prod_{i = m+1}^t e_p(\phi_i(\x)\partial_k P_i(y)) \right|^{2^{d+2}}\\
    &\leq \EE_{\x, y}\prod_{i=1}^{m}\prod_{\w\in\{0,1\}^{d+2}} \C^{|w|} g_i^{(\w)}(\x+\v_i P_i(y+w_1 k_1 + ... + w_{d+2} k_{d+2})) + O(p^{-1}).
\end{align*}
We fix this tuple. Letting $P_{i; \w}(y) = P_i(y+w_1 k_1 + ... + w_{d+2} k_{d+2})$, we rewrite the inequality above as 
\begin{align*}
    &\left|\EE_{\x,y,k} f_0(\x) \prod_{i=1}^{m}f_i(\x+\v_i P_i(y))\overline{g_i(\x+\v_i P_i(y+k))}\prod_{i = m+1}^t e_p(\phi_i(\x)\partial_k P_i(y)) \right|^{2^{d+2}}\\
    &\leq     \EE_{\x, y}\prod_{i=1}^{m}\prod_{\w\in\{0,1\}^{d+2}} \C^{|w|} g_{i}^{(\w)}(\x+\v_i P_{i; \w}(y))  + O(p^{-1}).
\end{align*}

For every $\w, \w'\in\{0,1\}^{d+2}$, the polynomials $P_{i; \w}$ and $P_{i; \w'}$ are distinct and equivalent. The polynomial family corresponding to the operator
\begin{align}\label{nice operator}
    \EE_{\x, y}\prod_{i=1}^{m}\prod_{\w\in\{0,1\}^{d+2}} \C^{|w|} g_{i}^{(\w)}(\x+\v_i P_{i; \w}(y)) 
\end{align}
is nice, which is a consequence of several facts:
\begin{enumerate}
    \item $\deg P_{m; \w} = \deg P_m > \deg P_i = \deg P_{i; \w'}$ for every $1\leq i < m$ and $\w, \w'\in\{0, 1\}^{d+2}$;
    \item $\deg (P_{i; \w}- P_{i; \w'}) = \deg P_i - 1$ for any distinct $\w, \w' \in\{0, 1\}^{d+2}$ and $1\leq i \leq m$; this follows from the fact that the expressions $(w_1 k_1 + ... + w_{d+2} k_{d+2})_{\w\in\{0,1\}^{d+2}}$ are all distinct;
    \item $\deg (P_{m; \w}- P_{m; \w'}) = \deg P_m - 1 > 0$ for distinct $\w, \w' \in\{0, 1\}^{d+2}$, which follows from the assumption $\deg P_m > 1$. 
\end{enumerate}
The properties (i) and (ii) in the definition of niceness follow from the first fact listed above; the property (iii) is a consequence of the other two facts. The proposition then follows from Proposition \ref{general Gowers norm control} applied to the operator (\ref{nice operator}). 


\end{proof}

\section{Degree lowering}\label{section on computations}
In this section, we fix integers $0\leq m\leq t$, nonzero vectors $\v_1, ..., \v_t\in\ZZ^D$, and polynomials $P_1, ..., P_t\in\ZZ[y]$ satisfying $0< \deg P_1 < \deg P_2 < ... < \deg P_t$. The main result of this section is the proposition below, from which we deduce Theorem \ref{count} in the next section. This result plays in our argument a similar part as Lemma 4.1 of \cite{peluse_2019} in that paper. 
\begin{proposition}\label{U^1 control of G}
There exists a constant $c>0$ with the following property: for all 1-bounded functions $f_1, ..., f_m:\FF_p^D\to\CC$, phase functions $\phi_{m+1}, ..., \phi_t:\FF_p^D\to\FF_p$ and subsets $\U\subset\FF_p^D$, the function
\begin{align*}
    G_{m,t}(\x) = \EE_y \prod_{i=1}^m f_i(\x+\v_i P_i(y)) \prod_{i=m+1}^t e_p(\phi_i(\x)P_i(y)) 1_\U(\x),
\end{align*}
satisfies
\begin{align*}
    \|G_{m,t}\|_{2}^2 = \EE_{\x}\left|\prod_{i=1}^m \EE(f_i|V_i)(\x)\right|^2 1_{\U'}(\x)+ O(p^{-2c}),
\end{align*}
where 
\begin{align*}
    \U'=\{\x\in \U: \phi_{m+1}(\x) = ... = \phi_t(\x)= 0\}.
\end{align*}
In particular, 
\begin{align*}
    \|G_{m,t}\|_{2}\leq \min_{1\leq i\leq m}\|f_i\|_{U^1(\v_i)} + O(p^{-c}),
\end{align*}
if $t\geq 1$, and if $\|G_{m,t}\|_{2}\geq \delta\gg p^{-c}$, then $|\U'|=\Omega(\delta^2 p^D)$.

\end{proposition}

We prove Proposition \ref{U^1 control of G} by induction on $m$. We start with the base case $m=0$. If $t=0$, then the statement is trivially true, otherwise it follows from Lemma \ref{Weyl sums}. The proof for $m\in\NN_+$ requires several technical lemmas which concern the properties of the dual function $F_{m,t}$.

\begin{lemma}\label{degree lowering of the dual}
Let $m\geq 1$, ${f_1, ..., f_m:\FF_p^D\to\CC}$ be 1-bounded, $\phi_{m+1}, ..., \phi_t:\FF_p^D\to\FF_p$ and $\U\subseteq\FF_p^D$. Let 
\begin{align*}
    F_{m,t}(\x) = &\EE_{y,k} \left(\prod_{i=1}^{m-1}f_i(\x+\v_i P_i(y)-\v_m P_m(y+k))\overline{f_i(\x+\v_i P_i(y+k)-\v_m P_m(y+k))}\right)\\
    &f_m(\x+\v_m (P_m(y)- P_m(y+k)))\left(\prod_{i=m+1}^t e_p(\phi_i(\x -\v_m P_m(y+k)) \partial_k P_i(y))\right)1_\U(\x-\v_m P_m(y+k)).
\end{align*}
For each integer $s>1$, there exists $c>0$ independent of the choice of functions $f_i, \phi_i$ and the set $\U$, for which
\begin{align*}
    \|F_{m,t}\|_{U^s(\v_m)}\ll \|F_{m,t}\|_{U^{s-1}(\v_m)}^{c} + p^{-c}. 
\end{align*}
\end{lemma}

Lemma \ref{degree lowering of the dual} plays an analogous role in our argument to Proposition 6.6 in \cite{peluse_prendiville_2019} and Lemma 8 in \cite{kuca_2020a} in corresponding papers. Multiple applications of Lemma \ref{degree lowering of the dual} and the H\"{o}lder inequality give the following corollary.
\begin{lemma}\label{full degree lowering of F}
Let $F_{m,t}$ be as in Lemma \ref{degree lowering of the dual}. For every $s\in\NN_+$, there exists a constant $c>0$ independent of the choice of functions $f_i, \phi_i$ and the set $\U$, for which
\begin{align*}
    \|F_{m,t}\|_{U^s(\v_m)}\ll \|F_{m,t}\|_{U^1(\v_m)}^{c} + p^{-c}. 
\end{align*}
\end{lemma}
\begin{proof}
The statement is trivially true for $s=1$, so suppose that $s>1$. Applying Lemma \ref{degree lowering of the dual}, we obtain that 
\begin{align}\label{bound on F, 1}
    \|F_{m,t}\|_{U^s(\v_m)}\ll \|F_{m,t}\|_{U^{s-1}(\v_m)}^{c_0} + p^{-c_0}. 
\end{align}
By induction hypothesis, there exists $c_1>0$ for which 
\begin{align}\label{bound on F, 2}
    \|F_{m,t}\|_{U^s(\v_m)}\ll \|F_{m,t}\|_{U^1(\v_m)}^{c_1} + p^{-c_1}. 
\end{align}
Combining (\ref{bound on F, 1}) and (\ref{bound on F, 2}) with the H\"{o}lder inequality, we get the result with $c = c_0 c_1$. 

\end{proof}

Finally, we show that the $\|F_{m,t}\|_{U^1(\v_m)}$ norm is bounded by the norms $\|f_1\|_{U^1(\v_1)}$, ..., $\|f_m\|_{U^1(\v_m)}$, a result analogous to Lemma 9 of \cite{kuca_2020a}. 
\begin{lemma}\label{U^1 control of F}
Let $F_{m,t}$ be as in Lemma \ref{degree lowering of the dual}. There exists a constant $c>0$ independent of the choice of functions $f_i, \phi_i$ and the set $\U$, for which $\|F_{m,t}\|_{U^1(\v_m)}\leq\min\limits_{1\leq i\leq m}\|f_i\|_{U^1(\v_i)} + O(p^{-c})$.
\end{lemma}

Our induction scheme works as follows. For $m\in\NN$, the $(m,t)$ case of Proposition \ref{U^1 control of G} is used to prove the $(m+1, t)$ cases of Lemmas \ref{degree lowering of the dual} and \ref{full degree lowering of F} as well as the $(m+1, t+1)$ case of Lemma \ref{U^1 control of F}. It follows that once the $(m,t)$ cases of Proposition \ref{U^1 control of G} are proved for all $t\geq m$, the $(m+1, t)$ cases of Lemmas \ref{degree lowering of the dual}, \ref{full degree lowering of F} and \ref{U^1 control of F} are proved for all $t\geq m+1$. The $(m+1,t)$ case of Proposition \ref{U^1 control of G} is then derived with the help of the $(m+1, t)$ cases of Lemmas \ref{full degree lowering of F} and \ref{U^1 control of F}.


\begin{proof}[Proof of Proposition \ref{U^1 control of G} in the case $m\geq 1$]
We recall that
\begin{align*}
    \|G_{m,t}\|^2_{2} = 
    \EE_{\x, y,k}\prod_{i=1}^{m}f_i(\x+\v_i P_i(y))\overline{f_i(\x+\v_i P_i(y+k))}\prod_{i = m+1}^t e_p(\phi_i(\x)\partial_k P_i(y)) 1_\U(\x).
\end{align*}

Translating $\x\mapsto \x-\v_m P_m(y+k)$, we observe that
\begin{align}\label{G as an inner product}
    \|G_{m,t}\|^2_{2} = \langle F_{m,t}, f_m\rangle,
\end{align}
where
\begin{align*}
    F_{m,t}(\x) = &\EE_{y,k} \left(\prod_{i=1}^{m-1}f_i(\x+\v_i P_i(y)-\v_m P_m(y+k))\overline{f_i(\x+\v_i P_i(y+k)-\v_m P_m(y+k))}\right)\\
    &f_m(\x+\v_m (P_m(y)- P_m(y+k)))
    \prod_{i = m+1}^t e_p(\phi_i(\x-\v_m P_m(y+k))\partial_k P_i(y))1_\U(\x-\v_m P_m(y+k)).
\end{align*}
is as in the statement of Lemma \ref{degree lowering of the dual}.
Applying the Cauchy-Schwarz inequality to (\ref{G as an inner product}), we obtain
\begin{align*}
    \|G_{m,t}\|^4_{2} \leq \|F_{m,t}\|^2_{2} = \EE_{\x, y,k}\prod_{i=1}^{m-1}f_i(\x+\v_i P_i(y))\overline{f_i(\x+\v_i P_i(y+k))}\\
    f_m(\x+\v_m P_m(y))\overline{F_{m,t}(\x+\v_m P_m(y+k))}\prod_{i = m+1}^t e_p(\phi_i(\x)\partial_k P_i(y)) 1_\U(\x).
\end{align*}
We thank Sean Prendiville for showing us the trick that we have just used to bound $\|G_{m,t}\|_{2}$ in terms of $\|F_{m,t}\|_{2}$. 

By Corollary \ref{Gowers norm control of twisted operators} applied to the sum above, there exists $s\in\NN_+$ and $0<c_0<1$, independent from the choice of $f_1, ..., f_m, \phi_{m+1}, ..., \phi_t, \U$, such that 
\begin{align}\label{bound 1 on G}
    \|G_{m,t}\|^4_{2} \leq\|F_{m,t}\|_{U^s(\v_m)}^{c_0} + O(p^{-c_0}).
\end{align}

We then apply Lemma \ref{full degree lowering of F} and Lemma \ref{U^1 control of F}, to bound 
\begin{align}\label{bound 2 on G}
    \|F_{m,t}\|_{U^s(\v_m)}\ll \min\limits_{1\leq i\leq m}\|f_i\|_{U^1(\v_i)}^{c_1} + p^{-c_1}
\end{align}
for some $0<c_1<1$. Combining (\ref{bound 1 on G}) and (\ref{bound 2 on G}), letting $c = c_0 c_1/4$ and using the H\"{o}lder inequality, we get the bound
\begin{align}\label{bound on G, 2}
    \|G_{m,t}\|^2_{2}\ll \min\limits_{1\leq i\leq m}\|f_i\|_{U^1(\v_i)}^{2c} + p^{-2c}.
\end{align}

Splitting each $f_1, ..., f_{m}$ into $f_i = \EE(f_i|V_i) + (f_i-\EE(f_i|V_i))$, observing that $\EE(f_i|V_i)(\x+\v_i P_i(y)) = \EE(f_i|V_i)(\x)$ and using the bound (\ref{bound on G, 2}) as well as the identity $\|f_i - \EE(f_i|V_i)\|_{U^1(\v_i)} = \|\EE(f_i - \EE(f_i|V_i)|V_i)\|_{2} = 0$, we deduce that
\begin{align*}
    \|G_{m,t}\|_{2}^2 &= \EE_{\x}\left|\EE_y\prod_{i=1}^{m} \EE(f_i|V_i)(\x+\v_i P_i(y)) \prod_{i=m+1}^t e_p(\phi_i(\x)P_i(y))\right|^2  1_{\U}(\x)+ O(p^{-2c})\\
    &= \EE_{\x}\left|\prod_{i=1}^{m} \EE(f_i|V_i)(\x)\right|^2 \cdot\left|\EE_y \prod_{i=m+1}^t e_p(\phi_i(\x)P_i(y))\right|^2  1_{\U}(\x)+ O(p^{-2c})
\end{align*}
As a consequence of Lemma \ref{Weyl sums} applied to the inner average over $y$, we obtain
\begin{align*}
    \|G_{m,t}\|_{2}^2 &= \EE_{\x}\prod_{i=1}^{m} |\EE(f_i|V_i)(\x)|^2  1_{\U'}(\x)+ O(p^{-2c})
\end{align*}
for 
\begin{align*}
    \U'=\{\x\in \U: \phi_{m+1}(\x) = ... = \phi_t(\x)= 0\}.
\end{align*}

It follows from the 1-boundedness of $f_1, ..., f_m$ that $|\U'|\geq \delta^2 p^D$ whenever $\|G_{m,t}\|_{2}\geq \delta \gg p^{-c}$. The 1-boundedness of $f_1, ..., f_m$ and the H\"{o}lder inequality further imply that
\begin{align*}
    \|G_{m,t}\|_{2}\leq \min_{1\leq i \leq m}\|f_i\|_{U^1(\v_i)} + O(p^{-c}).
\end{align*}
\end{proof}

We now proceed to prove Lemma \ref{degree lowering of the dual}, which contains the bulk of the technicalities in this paper. 
\begin{proof}[Proof of Lemma \ref{degree lowering of the dual}]
We recall that
\begin{align*}
    F_{m,t}(\x) = \EE_{y,k} &\left(\prod_{i=1}^{m-1}f_i(\x+\v_i P_i(y)-\v_m P_m(y+k))\overline{f_i(\x+\v_i P_i(y+k)-\v_m P_m(y+k))}\right)\\
    &f_m(\x+\v_m (P_m(y)- P_m(y+k)))\prod_{i=m+1}^t e_p(\phi_i(\x -\v_m P_m(y+k)) \partial_k P_i(y))1_\U(\x-\v_m P_m(y+k)),
\end{align*}
For simplicity, we set $F=F_{m,t}$ and $f_0 = 1_\U$ as well as recall that $\v_0 = \textbf{0}$ and $P_0 = 0$, so that
\begin{align*}
    F(\x) = \EE_{y,k} &\left(\prod_{i=0}^{m-1}f_i(\x+\v_i P_i(y)-\v_m P_m(y+k))\overline{f_i(\x+\v_i P_i(y+k)-\v_m P_m(y+k))}\right)\\
    &f_m(\x+\v_m (P_m(y)- P_m(y+k)))\prod_{i=m+1}^t e_p(\phi_i(\x -\v_m P_m(y+k)) \partial_k P_i(y)).
\end{align*}
We let $\delta = \|F\|_{U^s(\v_m)}$. We also denote $\underline{h}=(h_1, ..., h_{s-2})$ and $\EE_\h = \EE_{\h\in\FF_p^{s-2}}$. From the induction formula for Gowers norms and the inequality (\ref{U^2 inverse}), we deduce that
\begin{align*}
    \delta^{2^s} = \|F\|_{U^s(\v_m)}^{2^s}=\EE_{\h}\|\Delta_{\h \v_m} F\|_{U^2(\v_m)}^4\leq\EE_{\h,\x}|\widehat{\Delta_{\h \v_m} F}(\x; \v_m; \phi_m(\x; \h))|^2
\end{align*}
for some $\phi_m(\x; \h)\in\FF_p$. We can assume that $\phi_m(\x; \h)$ is the same for all $\x$ lying in the same coset of $V_m$ since $|\widehat{\Delta_{\h \v_m} F}(\x; \v_m; k)| = |\widehat{\Delta_{\h \v_m} F}(\x; \v_m; k + n\v_m)|$ for all $k, n\in\FF_p$. Thus, $\phi_m(\cdot; \h)$ is $V_m$-measurable for each fixed $\h$.
We let
\begin{align*}
    H_1=\{\h\in\FF_p^{s-2}: \|\Delta_{\h \v_m} F\|_{U^2(\v_m)}^4 \geq\delta^{2^s}/2\}
\end{align*}
and
\begin{align*}
    \U_\h = \{\x\in\FF_p^D: |\widehat{\Delta_{\h \v_m} F}(\x; \v_m; \phi_m(\x; \h))|^2\geq \delta^{2^s}/4\}.
\end{align*}
It follows from the popularity principle that 
\begin{align*}
    \delta^{2^s} &\ll \EE_{\x,\h} |\widehat{\Delta_{\h \v_m}F}(\x; \v_m; \phi_m(\x;\h))|^2 1_{\U_\h}(\x)1_{H_1}(\h)\\
    &\leq \EE_{\substack{\x,\h,\\ n,n'}}\Delta_{\h\v_m} F(\x+n\v_m)\overline{\Delta_{\h\v_m}F(\x+n'\v_m)}e_p(\phi_m(\x;\h)(n'-n))1_{\U_\h}(\x)1_{H_1}(\h).
\end{align*}
After expanding $\Delta_{\h\v_m} F(\x+n\v_m)$ and $\overline{\Delta_{\h\v_m}F(\x+n'\v_m)}$, the right-hand side of the above equals 
\begin{align*} 
    &\EE_{\x, n, n', \hh}\EE_{\yy,\yy', \kk,\kk'\in\FF_p^{\{0,1\}^{s-2}}}\prod_{\ww\in\{0,1\}^{s-2}} \C^{|w|}\left[\left(\prod_{i=0}^{m-1}f_i(\x+\textbf{v}_i P_i(y_\w)+\textbf{v}_m(n+\w\cdot\h -P_m(y_\w+k_\w)))\right. \right.\\
    &\overline{f_i(\x+\textbf{v}_i P_i(y_\w+k_\w)+\textbf{v}_m(n+\w\cdot\h -P_m(y_\w+k_\w)))}\overline{f_i(\x+\textbf{v}_i P_i(y'_\w)+\textbf{v}_m(n'+\w\cdot\h -P_m(y'_\w+k'_\w)))}\\
    &\left.\vphantom{\prod_{i=1}^{m-1}} f_i(\x+\textbf{v}_i P_i(y'_\w+k'_\w)+\textbf{v}_m(n'+\w\cdot\h -P_m(y'_\w+k'_\w)))\right)\\
    &f_m(\x+\textbf{v}_m(n+\w\cdot\h + P_m(y_\w) -P_m(y_\w+k_\w)))\overline{f_m(\x+\textbf{v}_m(n'+\w\cdot\h + P_m(y'_\w) -P_m(y_\w'+k_\w')))}\left.\vphantom{\prod_{i=1}^{m-1}}\right]\\
    &e_p\left(\phi_m(\x; \h)(n'-n)+\sum\limits_{i=m+1}^t\sum\limits_{\w\in\{0,1\}^{s-2}}(-1)^{|w|}(\phi_i(\x+\textbf{v}_m(n+\w\cdot\h-P_m(y_\w+k_\w)))\partial_{k_\w}P_i(y_\w)\right.\\
    &-\phi_i(\x+\textbf{v}_m(n'+\w\cdot\h-P_m(y'_\w+k'_\w)))\partial_{k'_\w}P_i(y'_\w)\left.\vphantom{\prod_{i=1}^{m-1}}\right)1_{\U_\h}(\x)1_{H_1}(\h).
\end{align*}
It does not suit us that the expression above contains a product of many copies of $f_1, ..., f_m$ whose arguments include different $y, y', k, k'$ variables. We want all the copies of $f_0, ..., f_m$ to be expressed in the same $y, y', k, k'$ variables. We shall achieve this by applying the Cauchy-Schwarz inequality $s-2$ times to the expression above. Letting $\tilde{\h} = (h_1, ..., h_{s-3}, h_{s-2}')$ and applying the Cauchy-Schwarz inequality in all variables except $h_{s-2}$, we bound the expression above by the square root of
\begin{align*}
    &\EE_{\substack{\x, n, n', h_1,\\ ..., h_{s-3}, h_{s-2}, h'_{s-2}}}\EE_{\substack{\yy,\yy', \kk,\\ \kk'\in\FF_p^{\{0,1\}^{s-2}}}}\prod_{\substack{\ww\in\{0,1\}^{s-2},\\ w_{s-2} = 1}} \C^{|w|}\left(\prod_{i=0}^{m}\tilde{f}_i\right)
    e_p\left(  \vphantom{\prod_{i=1}^{m-1}}  (\phi_m(\x; \h)-\phi_m(\x; \tilde{\h}))(n'-n)\right.\\
    &+\sum\limits_{i=m+1}^t\sum\limits_{\substack{\w\in\{0,1\}^{s-2}\\ w_{s-2} = 1}}(-1)^{|w|}((\phi_i(\x+\textbf{v}_m(n+\w\cdot\h-P_m(y_\w+k_\w)))\\
    &-\phi_i(\x+\textbf{v}_m(n+\w\cdot\tilde{\h}-P_m(y_\w+k_\w))))\partial_{k_\w}P_i(y_\w) -(\phi_i(\x+\textbf{v}_m(n'+\w\cdot{\h}-P_m(y_\w+k_\w)))\\
    &-\phi_i(\x+\textbf{v}_m(n'+\w\cdot\tilde{\h}-P_m(y_\w+k_\w))))\partial_{k'_\w}P_i(y'_\w))\left.\vphantom{\prod_{i=1}^{m-1}}\right) 1_{\U_\h}(\x)1_{H_1}(\h)1_{H_1}(\tilde{\h}),
\end{align*}
where 
\begin{align*}
    &\Tilde{f}_i(\x, \y, \k, n, n', \w, h_1, ..., h_{s-3}, h_{s-2}, h_{s-2}')\\
    &= f_i(\x+\v_i P_i(y_\w)+\v_m(n+\w\cdot\h -P_m(y_\w+k_\w)))\overline{f_i(\x+\v_i P_i(y_\w)+\v_m(n+\w\cdot\tilde{\h} -P_m(y_\w+k_\w)))}\\
    &\overline{f_i(\x+\v_i P_i(y_\w+k_\w)+\v_m(n+\w\cdot\h -P_m(y_\w+k_\w)))}f_i(\x+\v_i P_i(y_\w+k_\w)+\v_,(n+\w\cdot\tilde{\h} -P_m(y_\w+k_\w)))\\
    &\overline{f_i(\x+\v_i P_i(y'_\w)+\v_m(n'+\w\cdot\h -P_m(y'_\w+k'_\w)))}f_i(\x+\v_i P_i(y'_\w)+\v_m(n'+\w\cdot\tilde{\h} -P_m(y'_\w+k'_\w)))\\
    &f_i(\x+\v_i P_i(y'_\w+k'_\w)+\v_m(n'+\w\cdot\h -P_m(y'_\w+k'_\w)))\overline{f_i(\x+\v_i P_i(y'_\w+k'_\w)+\v_m(n'+\w\cdot\tilde{\h} -P_m(y'_\w+k'_\w)))}
\end{align*}
for $0\leq i\leq m-1$ and
\begin{align*}
    &\Tilde{f}_m(\x, \y, \k, n, n', \w, h_1, ..., h_{s-3}, h_{s-2}, h_{s-2}')\\
    &= f_m(\x+\v_m(n+\w\cdot\h + P_m(y_\w) - P_m(y_\w+k_\w)))\overline{f_m(\x+\v_m(n+\w\cdot\tilde{\h} + P_m(y_\w) - P_m(y_\w+k_\w)))}\\
    & \overline{f_m(\x+\v_m(n'+\w\cdot\h + P_m(y'_\w) - P_m(y'_\w+k'_\w)))}{f_m(\x+\v_m(n'+\w\cdot\tilde{\h} + P_m(y'_\w) - P_m(y'_\w+k'_\w)))}.
\end{align*}

Applying the Cauchy-Schwarz inequality another $s-3$ times, each time in all variables except $h_{s-3}$ during the first application, $h_{s-4}$ during the second application, etc., we obtain the bound
\begin{align*}
    \delta^{2^{2s-2}} \ll &\EE_{\x, \h,\h'}\left|\EE_{n, y, k}\left(\prod_{i=0}^{m-1} g_i(\x+\textbf{v}_i P_i(y) + \textbf{v}_m n)\overline{ g_i(\x+\textbf{v}_i P_i(y+k) + \textbf{v}_m n)}\right)\right.\\ 
    & g_m(\x+\textbf{v}_m(n+P_m(y)))e_p\left( \vphantom{\prod_{i=1}^{m-1}} \psi_m(\x; \hh,\hh')(n + P_m(y+k))\right.\\
    &+\sum_{i=m+1}^t\psi_i(\x + \textbf{v}_m n; \h, \h')\partial_k P_i(y)\left.\vphantom{\prod_{i=1}^{m-1}}\right)\left.\vphantom{\prod_{i=1}^{m-1}}\right|^2 1_{\U_{(\h,\h')}}(\x) 1_{H_2}(\h,\h')
\end{align*}
where
\[ \hh^{(\ww)}_i =
\begin{cases}
h_i, w_i = 0\\
h'_i, w_i = 1
\end{cases}
\]
\begin{align*}
    g_i(\x):=\prod_{\ww\in\{0,1\}^{s-2}}\C^{|w|}f_i\left(\x+\v_m \underline{1}\cdot\hh^{(\ww)}\right)\quad {\rm{with}}\quad \underline{1}=(1, ..., 1)\in\FF_p^{s-2},
\end{align*}
\begin{align*}
    H_2=\{(\hh,\hh')\in\FF_p^{2(s-2)}:\forall \ww\in\{0,1\}^{s-2}\hh^{(\ww)}\in H_1\} \quad {\rm{and}}\quad \U_{(\h, \h')} = \bigcap_{\w\in\{0,1\}^{s-2}}\U_{\h^{(\w)}},
\end{align*}
\begin{align*}
    \psi_i(\x; \h, \h') =\begin{cases}
    \sum\limits_{\ww\in\{0,1\}^{s-2}}(-1)^{|w|}\phi_m(\x;\hh^{(\ww)})\; &{\rm{for}}\; i=m\\
    \sum\limits_{w\in\{0,1\}^{s-2}}(-1)^{|w|}\phi_i(\x+\v_m\underline{1}\cdot \hh^{(\ww)})\; &{\rm{for}}\; m+1\leq i\leq t.
    \end{cases} 
\end{align*}
Applying the Cauchy-Schwarz inequality in $y, n$ to the expectation inside the absolute value, performing minor changes of variables and recalling that $\phi_m(\cdot, \h)$ and $\U_{(\h,\h')}$ are $V_m$-measurable, we get the bound
\begin{align*}
    \delta^{2^{2s-2}}\ll &\EE_{\x, \h,\h', n, y, k}\left(\prod_{i=0}^{m-1} g_i(\x+\textbf{v}_i P_i(y) + \textbf{v}_m n)\overline{ g_i(\x+\textbf{v}_i P_i(y+k) + \textbf{v}_m n)}\right)\\ 
    &e_p\left(\sum_{i=m}^t\psi_i(\x + \textbf{v}_m n; \h, \h')\partial_k P_i(y)\right) 1_{\U_{(\h,\h')}}(\x)1_{H_2}(\h,\h')\\
    &=\EE_{\x, \h,\h'}\left|\EE_y\prod_{i=0}^{m-1} g_i(\x+\textbf{v}_i P_i(y))e_p\left(-\sum_{i=m}^t\psi_i(\x; \h, \h')P_i(y)\right)\right|^2 1_{\U_{(\h,\h')}}(\x)1_{H_2}(\h,\h').
\end{align*}
We then use the 1-boundedness of $g_0$ and the fact that $g_0(\x + \v_0 P_0(y)) = g_0(\x)$ is independent of $y$ to conclude that
\begin{align*}
    \delta^{2^{2s-2}}\ll \EE_{\x, \h,\h'}\left|\EE_y\prod_{i=1}^{m-1} g_i(\x+\v_i P_i(y))e_p\left(-\sum_{i=m}^t\psi_i(\x; \h, \h')P_i(y)\right)\right|^2
    1_{\U_{(\h,\h')}}(\x)1_{H_2}(\h,\h').
\end{align*}

By the popularity principle, the set 
\begin{align*}
    H_3 = \left\{(\h,\h')\in H_2:\;  \EE_{\x}\left|\EE_y\prod_{i=1}^{m-1} g_i(\x+\v_i P_i(y))e_p\left(-\sum_{i=m}^t\psi_i(\x; \h, \h')P_i(y)\right)\right|^2 1_{\U_{(\h,\h')}}(\x)\gg \delta^{2^{2s-2}}\right\}
\end{align*}
has $\Omega(\delta^{2^{2s-2}} p^{2s - 4})$ elements. In particular, there exists $\h\in \FF_p^{s-2}$ for which the fiber
\begin{align*}
    H_4:=\{\hh': (\hh,\hh')\in H_3\}
\end{align*}
has $\Omega(\delta^{2^{2s-2}} p^{s - 2})$ elements. We fix such $\h$. 


Applying Proposition \ref{U^1 control of G} in the case $(m-1, t)$, we conclude that for each $\h'\in H_4$, the set
\begin{align*}
    \U'_{\h'}=\{\x\in\U_{(\h, \h')}: \psi_m(\x; \h, \h') = 0\}
\end{align*}
has $\Omega(\delta^{2^{2s-2}} p^D)$ elements as long as $\delta\gg p^{-c_1}$ for a constant $c_1>0$ given by the case $(m-1, t)$ of Proposition \ref{U^1 control of G}.

We now show that the phases $\psi_m$ possess some linear structure that we subsequently use to complete the proof. We define
\begin{align*}
    \eta_i(\x; \hh,\hh'):=(-1)^{s-1}\sum_{\substack{\ww\in\{0,1\}^{s-2},\\ w_1 = ... = w_{i-1}=1,\\ w_i = 0}}(-1)^{|w|}\phi(\x; \hh^{(\ww)}),
\end{align*}
so that $$\psi_m(\x; \hh,\hh')=(-1)^s\left(\phi_m(\x; \h')-\eta_1(\x;\hh,\hh')-...-\eta_{s-2}(\x; \hh,\hh')\right).$$ Crucially, $\eta_i$ does not depend on $h'_i$. Thus, $\psi(\x; \hh,\hh') = 0$ implies that
\begin{align*}
    \phi_m(\x; \h')=\sum_{i=1}^{s-2}\eta_i(\x; \hh,\hh').
\end{align*}
That is to say, $\phi_m(\x; \h')$ can be decomposed into a sum of $s-2$ functions, each of which does not depend on $h_i'$ for a different $i$.

We illustrate the aforementioned definitions for $s=3$ and 4. For $s=3$, 
$$\psi_m(
\x; h,h')=\phi_m(\x; h)-\phi_m(\x; h')=\eta_1(\x; h)-\phi_m(\x; h').$$ Hence $\psi_m(\x; h,h')=0$ implies that $\phi_m(\x; h')=\phi_m(\x; h)$. For $s=4$, 
\begin{align*}
    \psi_m(\x; \hh,\hh') &=\phi_m(\x; h_1,h_2)-\phi_m(\x; h_1,h_2')-\phi_m(\x; h_1',h_2)+\phi_m(\x; h_1',h_2')\\
    &=-\eta_1(\x; \hh,\hh')-\eta_2(\x; \hh,\hh')+\phi_m(\x; h_1',h_2')
\end{align*}
 and so $\psi_m(\x; \hh,\hh')=0$ implies that
\begin{align*}
    \phi_m(\x; h_1',h_2') = -\phi_m(\x; h_1,h_2) + \phi_m(\x; h_1,h_2') +\phi_m(\x; h_1',h_2)  = \eta_1(\x; \hh,\hh')+\eta_2(\x; \hh,\hh').
\end{align*}

To bound the $U^s(\v_m)$ norm of $F$ by its $U^{s-1}(\v_m)$ norm, we estimate the expression
\begin{align}\label{Long expression 7}
    \EE_{\hh', \x}\left|\widehat{\Delta_{\hh'\v_m}F}(\x; \v_m; \phi_m(\x;\hh'))\right|^2 1_{\U'_{\h'}}(\x)\1_{H_4}(\hh')
\end{align}
from above and below. For each $\h'\in H_4$ and $\x\in\U'_{\h'}$, we have $\left|\widehat{\Delta_{\hh'\v_m}F}(\x; \v_m; \phi_m(\x;\hh'))\right|^2\gg \delta^{2^s}$. Together with the lower bounds on the size of $H_4$ and $\U'_{\h'}$ whenever $\h'\in H_4$, we deduce that (\ref{Long expression 7}) is bounded from below by $\Omega(\delta^{4^s})$.

The upper bound is more complicated, and it relies on the fact that we can decompose $\phi_m(\hh')$ into a sum of $\eta_i$'s such that $\eta_i$ does not depend on $h'_i$. From the definitions of $H_4$ and $\U'_{\h'}$ it follows that
\begin{align}\label{Long expression 8}
    \nonumber
    &\EE_{\hh', \x}\left|\widehat{\Delta_{\hh'\v_m}F}(\x; \v_m; \phi_m(\x;\hh'))\right|^2 1_{\U'_{\h'}}(\x)\1_{H_4}(\hh')\\
    &=\EE_{\hh', \x}\left|\widehat{\Delta_{\hh'\v_m}F}\left(\x; \v_m; \sum_{i=1}^{s-2}\eta_i(\x;\hh, \hh')\right)\right|^2 1_{\U'_{\h'}}(\x)\1_{H_4}(\hh').
\end{align}
By positivity, we can extend (\ref{Long expression 8}) to the entire $\FF_p^{s-2}$; that is, we have
\begin{align*}
    &\EE_{\hh', \x}\left|\widehat{\Delta_{\hh'\v_m}F}\left(\x; \v_m; \sum_{i=1}^{s-2}\eta_i(\x;\hh, \hh')\right)\right|^2 1_{\U'_{\h'}}(\x)\1_{H_4}(\hh')\leqslant \EE_{\hh', \x} \left|\widehat{\Delta_{\hh'\v_m}F} \left(\x; \v_m; \sum_{i=1}^{s-2}\eta_i(\x;\hh, \hh')\right)\right|^2.
\end{align*}
Rewriting, we obtain that
\begin{align}\label{Long expression 9}
    \nonumber
    \EE_{\hh', \x} \left|\widehat{\Delta_{\hh'\v_m}F} \left(\x; \v_m; \sum_{i=1}^{s-2}\eta_i(\x;\hh, \hh')\right)\right|^2 &= \EE_{\hh', \x}\left|\EE_n\Delta_{\hh'\v_m}F(\x+n\v_m)e_p\left(\sum_{i=1}^{s-2}\eta_i(\x;\h, \hh')n\right)\right|^2\\
    &=\EE_{\hh',\x, n, k} \Delta_{\hh'\v_m, k\v_m} F(\x+n\v_m) e_p\left(\sum_{i=1}^{s-2}\eta_i(\x; \h, \hh')k\right).
\end{align}
We apply the Cauchy-Schwarz inequality $s-2$ times to (\ref{Long expression 9}) to get rid of the phases $\eta_i(\x; \h, \hh')$. In the first application, we apply the inequality in all variables but $h_1'$, thus bounding (\ref{Long expression 9}) by
\begin{align}\label{Long expression 10}
&\EE_{\substack{h_2',..., h'_{s-2},\\ \x, n, k}}\left|\EE_{h'_1}\Delta_{\substack{h'_2\v_m,...,h'_{s-2}\v_m, k\v_m}}F(\x+\v_m(n+h'_1))e_p\left(\sum_{i=2}^{s-2}\eta_i(\x; \h, \hh')k\right)\right|^2\\
\nonumber
&=\EE_{\substack{h_1', h_1'',\\ h_2',..., h'_{s-2},\\ \x, n, k}}\Delta_{h'_2\v_m,...,h'_{s-2}\v_m,k\v_m}\left(F(\x+\v_m(n+ h'_1))\overline{F(\x+\v_m(n + h''_1))}\right)\\
\nonumber
&e_p\left(\sum_{i=2}^{s-2}(\eta_i(\x; \h,\h')-\eta_i(\x; \h, \tilde{\h'}))k\right))^{\frac{1}{2}},
\end{align}
where $\tilde{\h'}=(h_1'', h_2', ..., h_{s-2}')$. After repeatedly applying the Cauchy-Schwarz inequality in this manner, we get rid of all the phases and bound (\ref{Long expression 10}) by $||F||^2_{U^{s-1}(\v_m)}$. Thus, $||F||_{U^{s-1}(\v_m)}\gg\delta^{2^{2s-1}}$ as long as $\delta\gg p^{-c_1}$. Taking $c=\min(c_1, 1/2^{2s-1})$, it follows that
\begin{align*}
    ||F||_{U^{s}(\v_m)}\ll ||F||_{U^{s-1}(\v_m)}^{c} + p^{-c}.
\end{align*}
\end{proof}

\begin{proof}[Proof of Lemma \ref{U^1 control of F}]
We set $F=F_{m,t}$. By definition,
\begin{align*}
    \|F\|_{U^1(\v_m)}^2 = \EE_\x \left|\EE_{\x+V_m}F\right|^2,
\end{align*}
where
\begin{align}\label{average of F on a coset}
    \nonumber
    \EE_{\x+V_m}F = &\EE_{n}F(\x + \v_m n) = \EE_{n, y, k}\left(\prod_{i=1}^{m-1} f_i(\x+ \v_i P_i(y)+\v_m n)\overline{f_i(\x+ \v_i P_i(y+k)+\v_m n)}\right)\\
    &f_m(\x + \v_m(n+P_m(y)))\prod_{i=m+1}^t e_p(\phi_i(\x +\v_m n) \partial_k P_i(y))1_\U(\x+\v_m n).
\end{align}

We first prove the statement when $m=t=1$. In that case,
\begin{align*}
    \|F\|_{U^1(\v_m)}^2 \leq \EE_{\x, n, y, k} \overline{f_m(\x+\v_m (n+P_m(y)))}{f_m(\x+\v_m (n+P_m(y+k)))}.
\end{align*}
Replacing both instances of $f_m$ by their Fourier series along $\v_m$, we observe that
\begin{align*}
    \|F\|_{U^1(\v_m)}^2 &\leq \EE_{\x}\sum_{l_1, l_2}  \overline{\widehat{f_m}(\x; \v_m; l_1)}\widehat{f_m}(\x; \v_m; l_2) \EE_{n, y, k} e_p(n(l_2-l_1) + l_2 P(y+k) - l_1 P(y)).
\end{align*}
Using Lemma \ref{Weyl sums} and Parseval's identity $\sum\limits_{l}  \left|\widehat{f_m}(\x; \v_m; l)\right|^2 = \EE_{\x+V_m}|f_m|^2$, we deduce that 
\begin{align*}
    \|F\|_{U^1(\v_m)}^2 &\leq \EE_{\x}\left|{\widehat{f_m}(\x; \v_m; 0)}\right|^2 + O(p^{-1/2}) = \|f_m\|_{U^1(\v_m)}^2 + O(p^{-1/2}).
\end{align*}

We assume now that $t>1$. Applying the Cauchy-Schwarz inequality in $k$ to (\ref{average of F on a coset}) and performing several changes of variables, we bound 
\begin{align*}
    \left|\EE_{\x+V_m}F\right|^2\leq \EE_{n}\left|\EE_y\prod_{i=1}^{m-1} f_i(\x+ \v_i P_i(y)+\v_m n) \prod_{i=m+1}^t e_p(-\phi_i(\x +\v_m n) P_i(y))1_\U(\x+\v_m n)\right|^2,
\end{align*}
and so 
\begin{align*}
    \|F\|_{U^1(\v_m)}^2 \leq \EE_{\x} \left|\EE_y\prod_{i=1}^{m-1} f_i(\x+ \v_i P_i(y)) \prod_{i=m+1}^t e_p(-\phi_i(\x) P_i(y))1_\U(\x)\right|^2.
\end{align*}
Applying the $(m-1, t-1)$ case of Proposition \ref{U^1 control of G} to 
\begin{align*}
    G_{m-1, t-1}(\x) = \EE_y\prod_{i=1}^{m-1} f_i(\x+ \v_i P_i(y)) \prod_{i=m+1}^t e_p(-\phi_i(\x) P_i(y))1_\U(\x),
\end{align*}
 which is where we use $t>1$, we deduce that
\begin{align}\label{control of F by U^1 norms of component functions}
    \|F\|_{U^1(\v_m)} \leq \min_{1\leq i\leq m-1}\|f_i\|_{U^1}(\v_i) + O(p^{-c_1})
\end{align}
for some $c_1>0$.

It remains to show that $\|F\|_{U^1(\v_m)}\leq \|f_m\|_{U^1}(\v_m) + O(p^{-c_1})$. Once again, we look at $\|F\|_{U^1(\v_m)}$, splitting each $f_1, ..., f_{m-1}$ into $f_i = \EE(f_i|V_i) + (f_i-\EE(f_i|V_i))$. Using (\ref{control of F by U^1 norms of component functions}) and the fact $\EE(f_i|V_i)(\x + \v_i P_i(y)) = \EE(f_i|V_i)(\x)$, we get that
\begin{align*}
    \|F\|_{U^1(\v_m)}^2 = &\EE_{\x} \left|\EE_{n,y,k}\prod_{i=1}^{m-1} |\EE(f_i|V_i)(\x+ \v_m n)|^2 f_m(\x+\v_m(n+P_m(y))) \right.\\
    &\left.\prod_{i=m+1}^t e_p(\phi_i(\x + \v_m n) \partial_k P_i(y))1_\U(\x + \v_m n)\right|^2 + O(p^{-c_1})
\end{align*}
We let $g(\x) = \prod\limits_{i=1}^{m-1} |\EE(f_i|V_i)(\x)|^2 1_\U(\x)$, so that
\begin{align*}
    \|F\|_{U^1(\v_m)}^2 &= \EE_{\x} \left|\EE_{n,y}g(\x+\v_m n) f_m(\x+\v_m(n+P_m(y))) \prod_{i=m+1}^t e_p(-\phi_i(\x+\v_m n)P_i(y))\right.\\
    &\left.\EE_k \prod_{i=m+1}^t e_p(\phi_i(\x+\v_m n) P_i(y+k))\right|^2.
\end{align*}
Using Lemma \ref{Weyl sums} and the linear independence of $P_{m+1}, ..., P_t$, the expectation in $k$ is of size $O(p^{-1/2})$ unless $\phi_{m+1}(\x+\v_m n) = ... =\phi_t(\x+\v_m n)=0$, and so
\begin{align*}
    \|F\|_{U^1(\v_m)}^2 &= \EE_{\x} |\EE_{n,y}g(\x+\v_m n) f_m(\x+\v_m(n+P_m(y)))|^2 + O(p^{-c})
\end{align*}
for $c = \min(c_1, 1/2)$.
To get rid of $g$, we apply the Cauchy-Schwarz inequality in $n$ to the inner expectation and obtain 
\begin{align*}
    \|F\|_{U^1(\v_m)}^2 &\leq \EE_{\substack{\x, y,\\ n, k}} \overline{f_m(\x+\v_m (n+P_m(y)))}{f_m(\x+\v_m (n+P_m(y+k)))} + O(p^{-c})
\end{align*}
We conclude the proof exactly the same way as in the $m=t=1$ case.

\end{proof}

\section{Estimating the number of progressions from below}\label{section on bounds}

With all the results from Section \ref{section on computations}, we are finally able to prove Theorem \ref{count}.
\begin{proof}[Proof of Theorem \ref{count}]
Let $f_0, ..., f_t: \FF_p^D\to\CC$ be 1-bounded, $P_1, ..., P_t\in\ZZ[y]$ be polynomials satisfying $0 < \deg P_1 < ... < P_t$, and $\v_1, ..., \v_t\in\ZZ^D$ be nonzero vectors. By Proposition \ref{U^1 control of G}, we have 
\begin{align*}
    \left|\EE_{\x, y}\prod_{i=0}^t f_i(\x+ \v_i P_i(y))\right|\leq\left(\EE_\x\left|\EE_{y}\prod_{i=0}^t f_i(\x+ \v_i P_i(y))\right|^2\right)^\frac{1}{2}\leq\min_{1\leq i\leq t}\|f_i\|_{U^1(\v_i)} + O(p^{-c})
\end{align*}
for a constant $c>0$ independent of $f_0, ..., f_t$. The statement follows by splitting each $f_1, ..., f_t$ as $f_i = \EE(f_i|V_i) + (f_i-\EE(f_i|V_i))$ and recalling that $\|f_i-\EE(f_i|V_i)\|_{U^1(\v_i)} = 0$ as well as $\EE(f_i|V_i)(\x + \v_i P_i(y)) = \EE(f_i|V_i)(\x)$.
\end{proof}

Corollary \ref{count for nonnegative functions} follows from the following lemma, which is a special case of Lemma 1.6 of \cite{chu_2011}.

\begin{lemma}\label{bounding the count from below}
Let $\v_1, ..., \v_t\in\ZZ^D$ be nonzero vectors and $f:\FF_p^D\to\CC$ be nonnegative. Then 
\begin{align*}
    \EE_{\x}\prod_{i=0}^t \EE(f|V_i)(\x) \geq \left(\EE_{\x\in\FF_p^D} f(\x) \right)^{t+1}.
\end{align*}
\end{lemma}

\begin{proof}[Proof of Theorem \ref{bounds}]
Suppose that $A\subseteq\FF_p^D$ has size $|A| = \alpha p^D$. Theorem \ref{bounds} and Lemma \ref{bounding the count from below} imply that there exists a constant $c>0$ for which $A$ contains $\Omega(\alpha^{t+1}p^D) - O(p^{D-c})$ nontrivial configurations (\ref{configuration}). It follows that if $\alpha\gg p^{-c/(t+1)}$, then $A$ contains a nontrivial configuration (\ref{configuration}).

\end{proof}

\bibliography{library}
\bibliographystyle{alpha}

\end{document}